\documentclass[12pt]{extarticle}
\usepackage{extsizes}
\usepackage[reqno]{amsmath} 
\usepackage{amsmath}
\usepackage{amssymb}
\usepackage{amsthm}
\usepackage{amscd}
\usepackage{amsfonts}
\usepackage{graphicx}
\usepackage{dsfont}
\usepackage{fancyhdr}
\usepackage{enumerate}
\usepackage{youngtab}
\usepackage[utf8]{inputenc}
\usepackage{comment}
\usepackage{mathrsfs}

\usepackage{subfigure}
\usepackage[margin=1.5in]{geometry}
\usepackage{graphicx}
\usepackage{algorithm}
\usepackage{algpseudocode}
\usepackage{color}
\usepackage{hyperref}
\usepackage{notation}

\theoremstyle{theorem}

\newtheorem{corollary}{Corollary}

\newtheorem{lemma}[corollary]{Lemma}
\newtheorem{lemma*}[lem6]{Lemma}

\newtheorem{theorem}[corollary]{Theorem}
\newtheorem{problem}{Problem}

\begin{document}

\AtEndDocument{%
  \par
  \medskip
  \begin{tabular}{@{}l@{}}%
    \textsc{Thiago Assis} \\
    \textsc{Dept. of Computer Science} \\ 
    \textsc{Universidade Federal de Minas Gerais, Brazil} \\
    \textit{E-mail address}: \texttt{thiago.assis@dcc.ufmg.br}\\ \ \\
    \textsc{Gabriel Coutinho}\\
    \textsc{Dept. of Computer Science} \\ 
    \textsc{Universidade Federal de Minas Gerais, Brazil} \\
    \textit{E-mail address}: \texttt{gabriel@dcc.ufmg.br} \\ \ \\
    \textsc{Emanuel Juliano} \\
    \textsc{Dept. of Computer Science} \\ 
    \textsc{Universidade Federal de Minas Gerais, Brazil} \\
    \textit{E-mail address}: \texttt{emanuelsilva@dcc.ufmg.br}
    
  \end{tabular}}

\title{Spectral upper bounds for the Grundy number of a graph}
\author{Thiago Assis \and Gabriel Coutinho \and Emanuel Juliano\footnote{emanuelsilva@dcc.ufmg.br --- remaining affiliations in the end of the manuscript.}}
\date{\today}
\maketitle
\vspace{-0.8cm}

\begin{abstract} 
The Grundy number of a graph is the minimum number of colors needed to properly color the graph using the first-fit greedy algorithm regardless of the initial vertex ordering. Computing the Grundy number of a graph is an NP-Hard problem. There is a characterization in terms of induced subgraphs: a graph has a Grundy number at least k if and only if it contains a k-atom. In this paper, using properties of the matching polynomial, we determine the smallest possible largest eigenvalue of a k-atom. With this result, we present an upper bound for the Grundy number of a graph in terms of the largest eigenvalue of its adjacency matrix. We also present another upper bound using the largest eigenvalue and the size of the graph. Our bounds are asymptotically tight for some infinite families of graphs and provide improvements on the known bounds for the Grundy number of sparse random graphs.
\end{abstract}

\begin{center}
\end{center}

\section{Introduction}\label{intro}

A $k$-\textit{coloring} of a graph $G$ is a surjective function $f : V(G) \to [k]$, where $[k] = \{1,\dots,k\}$, and $f(x) \neq f(y)$ whenever $x$ is adjacent to $y$. The \textit{chromatic number} $\chi(G)$ is the minimum $k$ such that $G$ admits a k-coloring. Given a graph $G$ and an ordering $\sigma = (x_1, x_2, \dots, x_n)$ of $V(G)$, the \textit{first-fit coloring algorithm} colors vertex $x_i$ with the smallest color that is not present among the neighbors of $x_i$ in $\{x_1 \dots, x_{i-1}\}$. The \textit{Grundy number} $\Gamma(G)$ is the largest $k$ such that the first-fit coloring algorithm generates a $k$-coloring of $G$.

Let $\Delta(G)$ denote the largest degree of $G$. It is well known (and immediate) that
\[
\chi(G) \leq \Gamma(G) \leq \Delta(G) + 1.
\]
Both parameters $\chi(G)$ and $\Gamma(G)$ are NP-hard to compute \cite{ZAKER20063166} and even determining if $\Gamma(G) \leq \Delta(G)$ is NP-complete \cite{HavetFredericGrundy}.

Let $A$ denote the adjacency matrix of a graph $G$. Let $\lambda_1(G) \geq \lambda_2(G) \geq \dots \geq \lambda_n(G)$ denote the eigenvalues of $A$. It is known (see for instance \cite[Chapter 3]{brouwer2011spectra}) that

\[
\sqrt{\Delta(G)} \leq \lambda_1(G) \leq \Delta(G). 
\]

Equality on the left holds if and only if $G$ is the complete bipartite graph $K_{1, n}$, and equality on the right holds if and only if $G$ is regular. Wilf \cite{Wilf1967TheEO} showed a sharper bound for the chromatic number:
\[
\chi(G) \leq 1 + \lambda_1(G).
\]
This upper bound holds with equality if and only if $G$ is complete or an odd cycle (but no proof of this is known that does not make use of Brooks theorem), and, of course, Wilf's bound can be trivially improved by noticing that $\chi(G)$ is always an integer, whereas $\lambda_1(G)$ might not be.

Our main result is an upper bound for the Grundy number only in terms of the largest eigenvalue of the graph, and this is presented in Section \ref{sec:2}. We prove that
\[
\Gamma(G) \leq \frac{(\lambda_1(G) + o(1))^2}{2},
\]
where $o(1)$ denotes a function that tends to 0 as $\lambda_1(G)$ increases. While our main result is not a strict improvement on $\Gamma(G) \leq \Delta(G) + 1$ for all graphs, it does offer improvement specially for sparse graphs, as we will discuss later on.

As a consequence of our work, we also relate the Grundy number and the largest root of the matching polynomial of $G$. Our methods involve interlacing, the connection between the characteristic and the matching polynomial due to Godsil and Gutmann \cite{godsil1981theory}, and an elementary calculation of the largest eigenvalue of a family of $k$-atoms (to be defined in the next section).

In Section \ref{sec:3}, we discuss how to incorporate the number of vertices $n$ of $G$ in a spectral upper bound for $\Gamma(G)$, and prove the inequality:
\[
\Gamma(G) \leq O\left(\lambda_1(G) \left(\frac{n}{\lambda_1(G)}\right)^{1/3}\right).
\]

In Section \ref{sec:4}, we compare our results with known bounds for the Grundy number of random graphs, giving new upper bounds for the Grundy number of random graphs with non-constant edge density.

\section{Grundy number and the largest root of the matching polynomial} \label{sec:2}

For each positive integer $k$, a class of graphs denoted by $\mathbf{\mathcal{A}_k}$ was introduced by Zeker \cite{ZAKER20063166}, the class satisfies the following property. If the Grundy number of a graph $G$ is at least $k$, then $G$ contains a subgraph isomorphic to an element of $\mathbf{\mathcal{A}_k}$. Any element of $\mathbf{\mathcal{A}_k}$ is called \textit{$k$-atom}. $K_1$ is the only 1-atom and $K_2$ is the only 2-atom. A graph $H$ is a $(k + 1)$-atom if its vertices can be partitioned into an independent set $I$ and a $k$-atom $A_k$, such that each vertex of $A_k$ has exactly one edge incident to $I$ and each element from $I$ has at least one edge incident to $A_k$. In this paper, we use Zeker's definition of $k$-atom, but slightly different definitions are present in the literature. For instance, Effantin et al. \cite{effantin2016characterization} define $H$ to be a $(k+1)$-atom if there is \textit{at least} one edge from each vertex of $A_k$ to $I$. We state this difference because we do not use the complete characterization of Grundy numbers in terms of $k$-atoms from \cite[Theorem 1]{ZAKER20063166}, just the observation that if $\Gamma(G) \geq k$, then $G$ contains a $k$-atom as a subgraph.

We say that the \textit{seed} of $A_k$ is a vertex that could have been the initial vertex during its recursive construction. For each $k$ there exists a unique tree that is a $k$-atom, denoted by $T_k$. We have that $|V(T_k)| = 2^{k-1}$ and that $T_k$ is the largest $k$-atom.

Given a graph $G$ and a vertex $u$, the \textit{path tree} of $(G,u)$ is a rooted tree which vertices are in bijection with the paths in $G$ that start at $u$. The root is identified with the unique path of length $0$, and two vertices are neighbors in the path tree if and only if one of the corresponding paths is obtained from the other upon adjoining one edge in the end. 

\begin{lemma} \label{lem:subgraph}
    For any $k$-atom $A_k$ with seed $s$, $T_k$ is a subgraph of the path tree of $(A_k,s)$, and is equal to this path tree if and only if $A_k = T_k$.
\begin{proof}
    We will prove by induction on $k$. The base case is simple: $T_1$ is a subgraph of the path tree of any $A_1$, which is only a single vertex.
    
    Let $A_k$ be any $k$-atom, with vertex partition $V(A_k) = V(A_{k-1}) \cup I$, where $I$ is an independent set and $A_{k-1}$ is a $(k-1)$-atom. Let $s$ be a seed of $A_{k-1}$, thus also one for $A_k$. We will denote by $T = T(A_k,s)$ the path tree of $A_k$ with respect to $s$, and similarly, $T'=T(A_{k-1},s)$ for $A_{k-1}$.

    Because each vertex of $A_{k-1}$ receives a new neighbour when constructing $A_k$, it follows that all paths starting at $s$ and within $A_{k-1}$ can be extended with at least one extra edge in their end in $A_k$. Therefore the tree $T$ contains the tree obtained from $T'$ by adding one new leaf to all of its vertices, and will be equal to it if and only if each path could have been extended by exactly one edge, case which corresponds to when $A_k$ was obtained from $A_{k-1}$ upon adding an independent set of size $|V(A_{k-1})|$ matched with $A_{k-1}$.

    By induction, $T_{k-1}$ is a subgraph of $T'$, thus the paragraph above implies $T_k$ is a subgraph of $T$, and the equality characterization also follows.
\end{proof}
\end{lemma}

The matching polynomial of a graph $G$ on $n$ vertices is defined as
\begin{equation}
    \mu_G(x) = \sum_{M} (-1)^{|M|} x^{n-2|M|},\label{matchingpol}
\end{equation}
where the sum runs over all $M \subseteq E$ which are matchings of the graph. Denote by $\phi_G(x)$ the characteristic polynomial of $A(G)$. The connection between the matching polynomial and the characteristic polynomial of a graph has been well understood and has found countless applications, including in the recent breakthrough \cite{interlacing_familiesI}. A graph $G$ is a forest if and only if $\mu_G = \phi_G$ (see \cite[Corollary 4.2]{godsil1981theory}), but more generally, the roots of $\mu_G$ are all roots of the characteristic polynomial of the path tree of $G$, and, for any graph $G$ and vertex $u$, the following formula holds (see \cite[Theorem 5]{godsil1981theory}):
    \begin{equation}
        \frac{\mu_G}{\mu_{G \setminus u}} = 
        \frac{\mu_{T(G,u)}}{\mu_{T(G,u) \setminus u}}. \label{pathtree}
    \end{equation}
For the result below, let $\mu_1(G) \ge \dots \geq \mu_n(G)$ denote the roots of the matching polynomial of a graph $G$. We will also use the fact that $\lambda_1(G) \geq \mu_1(G)$ for all graphs, and if $G$ is connected, equality holds if and only if $G$ is a tree (\cite[Theorem 6]{godsil1981theory}).

\begin{lemma}
    Let $G$ be a connected graph with $\Gamma(G) \geq k$. Then, 
    \[
        \mu_1(G) \geq \lambda_1(T_k),
    \] 
    and equality holds if and only if $G = T_k$. Moreover, $\lambda_1(G) \geq \lambda_1(T_k)$.
\end{lemma}
\begin{proof}
    Let $A_k$ be a $k$-atom contained in $G$. We will verify the following chain of inequalities.
    \begin{equation}
        \lambda_1(G) \geq 
        \mu_1(G) \geq 
        \mu_1(A_k) = 
        \mu_1(T(A_k,s)) \geq 
        \mu_1(T_k) =
        \lambda_1(T_k). \label{chain}
    \end{equation}
    The first inequality is \cite[Theorem 6]{godsil1981theory}.

    The largest root of $\mu_G$ is also the largest root of $\mu_{T(G,u)}$, and they are both simple roots. This follows from \eqref{pathtree}, the fact that $\mu_{T(G,u)} = \phi_{T(G,u)}$, and the Perron-Frobenius theory (see \cite[Section 2.2]{brouwer2011spectra}) which among many other things implies that the largest root of the characteristic polynomial of a connected graph is strictly larger than the roots of the characteristic polynomial of the subgraphs. So the first equality in \eqref{chain} holds.
    
    Also, as a consequence of these facts and interlacing (see \cite[Section 2.5]{brouwer2011spectra}), the largest root of $\mu_G$ is strictly larger than any root of  $\mu_{H}$ for any proper subgraph $H$. The reason for this is that the path tree $T(H, u)$ is a proper induced subgraph of the path tree $T(G, u)$, as the set of paths starting in $u$ in $H$ form a proper subset of the paths starting in $u$ in $G$. Thus, the second inequality in \eqref{chain} holds because $A_k$ is a subgraph of $G$, with equality if and only if $G = A_k$.
    
    The last inequality, and its equality characterization, is due to Lemma~\ref{lem:subgraph} and the paragraphs above. The last equality comes from the matching polynomial and characteristic polynomial of trees being equal.    
\end{proof}

The tree $T_k$ is sometimes called binomial tree \cite{cormen2022introduction} or binary hypertree \cite{BARRIERE200936}, the latter name is due to the fact that it can be represented as the hierarchical product of several copies of the complete graph on two vertices; Barrière et al \cite{BARRIERE200936} studied the spectrum of this family of trees and provide a characterization on the largest eigenvalue of $T_k$.

The largest eigenvalue of $T_k$ verify the recurrence

\begin{equation}
    \lambda_1(T_k) = \frac{1}{2} \left(\lambda_1(T_{k-1}) + \sqrt{\lambda_1(T_{k-1})^2 + 4}\right)\label{recurrence}.
\end{equation}

The asymptotic behavior of $\lambda_1(T_k)$ as $k \to \infty$ is:

\begin{equation*}
     \lambda_1(T_k) \sim \sqrt{2(k-1)}\label{asymptotic}.
\end{equation*}

The recurrence relation follows by induction looking at the eigenvectors of the adjacency matrix of $T_k$. The asymptotic behavior follows from the observation that the recurrence tends to a power law $\alpha k^{\beta}$ as $k \to \infty$, implying that the asymptotic behavior of $\lambda_1(T_k)$ tends to $\sqrt{2(k-1)}$. We provide an alternative proof for the asymptotic behavior of $\lambda_1(T_k)$ and bound it quite tightly.

\begin{lemma}\label{lem:TkBound}
    The largest eigenvalue of $T_k$ is bounded by
    \[
    \sqrt{2(k-1)} \geq \lambda_1(T_k) \geq \sqrt{2(k-1)}-o(1).
    \]
    where $o(1)$ is a function that goes to $0$ as $k$ grows.
\end{lemma}
\begin{proof}

Let $f_1 = 0$ and $f_{k+1} = \frac{1}{2}\left(f_k + \sqrt{f_k^2 + 4}\right)$ be the recurrence relation for the largest eigenvalue of $T_k$, following equation \eqref{recurrence}. First, we show that $f_k \leq \sqrt{2(k-1)}$. Notice that $f_1 = 0 \leq \sqrt{0}$. We proceed by induction,
\begin{align*}
f_{k+1} &= \frac{1}{2}\left(f_k + \sqrt{f_k^2 + 4}\right) \\
&\leq \frac{1}{2}\left(\sqrt{2(k-1)} + \sqrt{2(k-1) + 4}\right) \\
&\leq \sqrt{2k},
\end{align*}
where the last inequality follows from the fact that $\sqrt{x}$ is a concave function.


Now, we show that $f_k > \sqrt{2(k-1)} - \epsilon$ for any constant $\epsilon > 0$ and $k$ sufficiently large. First note that $f_k$ is an unbounded sequence of non-negative numbers, because the largest degree of $T_k$ is increasing and $f_k \geq \sqrt{\Delta(T_k)}$.

Consider $k$ large enough such that $k' := f_k^2/2+1$ satisfies 
\[\sqrt{2k'} - \sqrt{2(k'-1)} < \epsilon.\]
Let $d_k$ be the difference between $\sqrt{2(k-1)}$ and $f_k$:
\[d_{k} := \sqrt{2(k-1)} - f_k = \sqrt{2(k-1)} - \sqrt{2(k'-1)}.\]
We first notice that
\begin{align*}
f_{k+1} - f_k &= \frac{\sqrt{2(k'-1)} + \sqrt{2(k'-1) + 4}}{2} - \sqrt{2(k'-1)} \\
&= \frac{\sqrt{2(k'+1)} - \sqrt{2(k'-1)}}{2} \\
&> \sqrt{2(k'+1)} - \sqrt{2k'},
\end{align*}
where the last inequality follows from the fact that $\sqrt{x}$ is concave. 
Therefore,
\begin{equation*}
   d_{k+1} < d_{k} + \left(\sqrt{2k} - \sqrt{2(k-1)}\right) - \left(\sqrt{2(k'+1)} - \sqrt{2k'}\right).
\end{equation*}

Let $k'' := f_{k+1}^2/2$, and as $k'' < k' + 1$, we have that $\sqrt{2(k''+1)} - \sqrt{2k''} > \sqrt{2(k'+2)} - \sqrt{2(k'+1)}$ and the argument can be repeated. After $l$ steps we get:
\begin{align*}
d_{k+l} &< d_{k} + \left(\sqrt{2(k+l-1)} - \sqrt{2(k-1)}\right) - \left(\sqrt{2(k'+l)} - \sqrt{2k'}\right) \\
&= d_{k} + \left(\sqrt{2(k+l-1)} - \sqrt{2(k'+l)}\right) - \left(\sqrt{2(k-1)} - \sqrt{2k'}\right).
\end{align*}
For sufficiently large $l$, that is, for $l \to \infty$,
\begin{equation*}
\begin{split}    
d_{k+l} &< d_{k} + o(1) -  \left(\sqrt{2(k-1)} - \sqrt{2k'}\right) \\
&= o(1) + \sqrt{2k'} - \sqrt{2(k'-1)} \\
& < \epsilon.
\end{split}
\end{equation*}
\end{proof}

The previous result provide our desired upper bound for the Grundy number. We state it in terms of $\mu_1(G)$, but recall that $\mu_1(G) \leq \lambda_1(G)$.

\begin{theorem} \label{thm:grundy_bound}
Let $G$ be a graph with largest root of the matching polynomial $\mu_1(G)$, let $o(1)$ denotes a function that tends to 0 as $\mu_1(G)$ increases, then
\[
\Gamma(G) \leq \frac{(\mu_1(G)+o(1))^2}{2} + 1.
\] \qed
\end{theorem}

We point out that the term $o(1)$ above is easily shown to be always $\leq 1/2$.

\section{Spectral upper bound depending on the size of the graph} \label{sec:3}

Despite Theorem \ref{thm:grundy_bound} implying tight limits for the Grundy number of trees $T_k$, it does not provide much information when the largest eigenvalue is not close to the square root of the maximum degree. Therefore, a question that arises is whether we can incorporate the size of the graph into the bound so as to balance this out.

Wu and Elphick \cite{Elphick2015Achromatic} answer a question in this direction by finding an upper bound for the Grundy number in terms of the number of edges and the spectrum of the graph:
\[ \Gamma(G) \leq \frac{2|E(G)|}{\lambda_1(G)}. \]

Moreover, the Grundy number can also be bounded only in terms of the largest eigenvalue and the number of vertices by using the degeneracy of the graph. A graph $G$ is said to be $d$-degenerate if every subgraph has a vertex of degree at most $d$. The degeneracy of a graph is the smallest value of $d$ for which $G$ is $d$-degenerate. The Grundy number of $d$-degenerate graphs has been previously studied in \cite{irani1994coloring}, \cite{Balogh2008OnTF}, and \cite{chang2012first}. It has been proven that the Grundy number of $d$-degenerate graphs on $n$ vertices is bounded by:
\begin{equation*}
    \Gamma(G) \leq O(d \log n), 
\end{equation*}
and there are graphs that attain this bound.

This leads to an immediate spectral upper bound. Let $\delta(G)$ denote the smallest degree of $G$. It is known, as mentioned in \cite[Chapter 3]{brouwer2011spectra}, that
\begin{equation*}
    \delta(G) \leq \lambda_1(G). 
\end{equation*}
If the degeneracy of $G$ is $d$, then there exists a subgraph $H$ of $G$ with $\delta(H) = d$. Therefore, we can attain the spectral bound
\begin{equation}\label{eq:degeneracy}
    \Gamma(G) \leq O(d \log n) \leq O(\lambda_1(H) \log n) \leq O(\lambda_1(G) \log n). 
\end{equation}

The goal of this section is to establish a bound for the Grundy number that is not directly comparable to the inequality \eqref{eq:degeneracy}. Recall from Section \ref{intro} that a $k$-atom is constructed in $k$ steps, adding $a_i$ vertices at the $i$-th step. Let $S$ be the $n \times k$ normalized partition matrix whose rows correspond to vertices and columns to the sets of vertices added in each step. Then, $S^TS = I$ and we can define $B = S^TAS$ entrywise as
\begin{equation*}
    (B)_{ij} = \begin{cases}
        \sqrt{\frac{a_i}{a_j}}, & \text{if } i < j, \\
        0, & \text{if } i = j, \\
        \sqrt{\frac{a_j}{a_i}}, & \text{if } i > j.
    \end{cases}
\end{equation*}

By applying interlacing, we notice that

\[
\lambda_{1}(G) \geq \lambda_{1}(B) = \max_{x \in \mathds{R}^k} \frac{x^T B x}{x^Tx} \geq \frac{\1^TB\1}{k} = \frac{2}{k} \sum_{i < j} \sqrt{\frac{a_i}{a_j}}.
\]

Our goal for the rest of the section is to bound the right hand side of the inequality in terms of the Grundy number of the graph.

\begin{lemma}\label{lem:increasing}
    Let $n$ be fixed, and assume $\{a_i\}_{i = 1}^k$ is a possible sequence of sizes of sets used in the construction of a $k$-atom on $n$ vertices, so that $\sum a_i = n$. The sequence $a_1, \dots, a_k$ that achieves the minimum value for $\sum_{i < j} \sqrt{\frac{a_i}{a_j}}$ is non decreasing.
\end{lemma}
\begin{proof}
    Let $a_1, \dots, a_k$ be the sequence that achieves the minimum value, and assume by contradiction that exists an index $i < k$ such that $a_{i+1} < a_i$. If we exchange $a_i$ with $a_{i+1}$, the value of $\sum_{i < j} \sqrt{\frac{a_i}{a_j}}$ changes by
    \begin{equation*}
    \sqrt{\frac{a_{i+1}}{a_i}} - \sqrt{\frac{a_i}{a_{i+1}}} < 0. 
    \end{equation*}
\end{proof}

\begin{lemma}
    The largest eigenvalue of a $k$-atom $A_k$ with $n$ vertices is bounded by
    \[
    \lambda_1(A_k) \geq \frac{k\sqrt{k}}{4\sqrt{n}} - 2.
    \]
\end{lemma}
\begin{proof}

The proof follows from a chain of inequalities, starting with the observation that $a_i \geq 1$, thus

\begin{equation*}
    \sum_{i<j \leq k} \sqrt{\frac{a_i}{a_j}} \geq \sum_{i<j \leq k} \sqrt{\frac{1}{a_j}} = \sum_{j \leq k} \frac{j-1}{\sqrt{a_j}}
\end{equation*}

Since the sequence is non-decreasing (Lemma \ref{lem:increasing}), the first $0.8 \cdot k$ elements are bounded, i.e., $a_j \leq \frac{5n}{k}$, giving,

\begin{equation*}
\sum_{j \leq k} \frac{j-1}{\sqrt{a_j}} \geq \sqrt{\frac{k}{5n}} \sum_{j \leq 0.8 \cdot k} (j-1) \geq \frac{k^2 \sqrt{k}}{8\sqrt{n}} - k.\\
\end{equation*}
Therefore,
\begin{equation*}
    \lambda_{1}(G) \geq \lambda_{1}(B) \geq \frac{2}{k} \sum_{i < j} {\sqrt{\frac{a_i}{a_j}}} \geq \frac{k \sqrt{k}}{4\sqrt{n}} - 2.
\end{equation*}
\end{proof}

\begin{corollary} \label{corol:grundy_bound_2}
    The Grundy number of a graph with $n$ vertices is bounded by
    \[
    \Gamma(G) \leq O\left(\lambda_1(G) \left(\frac{n}{\lambda_1(G)}\right)^{1/3}\right).
    \]
\end{corollary}

\section{Grundy number of random graphs} \label{sec:4}
The Erdős–Rényi $\mathscr{G}(n, p)$ model of random graphs consists of all labeled graphs with vertex set $V = \{1, 2, \dots, n\}$ in which the edges are chosen independently and with probability $p$, where $0 < p < 1$, the value $p$ is sometimes called edge density. We refer to a random graph $G(n, p)$ as a graph on $n$ vertices generated according to the distribution $\mathscr{G}(n, p)$. We say that a graph property $A$ holds almost surely, or a.s. for brevity, in $G(n, p)$ if the probability that $G(n, p)$ has $A$ tends to one as the number of vertices $n$ tends to infinity \cite{random_graphs}. We say that a random graph is \textit{sparse} if the expected average degree is constant.

The Grundy number of random graphs was previously studied, particularly when the edge density $1/2$. For instance, Bollobás and Erdős \cite{Bollobas_Erdös_1976}, \cite[p. 294-298]{random_graphs} and McDiarmid \cite{mcdiarmid1979colouring} showed that a.s.,

\begin{equation}\label{eq:bollobas_bound}
    \Gamma(G) \leq \left(1 + \frac{5 \ln{\ln{n}}}{\ln{n}}\right) \frac{n}{\log_2{n}}.
\end{equation}

Our contribution in this section is to provide bounds for the Grundy number of random graphs where the edge density $p$ is a function on $n$ smaller than a constant. To achieve this goal, we first need to examine the asymptotic behaviour of the largest eigenvalue of random graphs. Krivelevich and Sudakov \cite{krivelevich_sudakov_2003} showed that the largest eigenvalue of a random graph $G(n, p)$ satisfies almost surely
\[\lambda_1(G) = (1 + o(1)) \max \{\sqrt{\Delta(G)}, np\}.\] 
This already gives an indication of for which random graphs Theorem \ref{thm:grundy_bound} derives a better bound than the maximum degree. Moreover, the same authors present the following result in \cite[Corollary 1.2]{krivelevich_sudakov_2003}, concerning sparse random graphs. For any constant $c > 0$, a.s. 
\begin{equation*}
    \lambda_1(G(n, c/n)) = (1 + o(1)) \sqrt{\frac{\ln n}{\ln \ln n}}.
\end{equation*}

This result, combined with Theorem \ref{thm:grundy_bound}, gives the following bound for the Grundy number of sparse random graphs:

\begin{corollary} \label{cor:sparse}
 For any constant $c > 0$, a.s
 \[
 \Gamma(G(n, c/n)) \leq \left(\frac{1}{2} + o(1)\right) \frac{\ln n}{\ln \ln n}.
 \]
\end{corollary}

We can also apply Corollary \ref{corol:grundy_bound_2} on random graphs with largest eigenvalue close to $(1+o(1))np$, a condition fulfilled when the edge density satisfies $1/\sqrt{n} \leq p$.

\begin{corollary} \label{cor:dense}
 For any edge density $1/\sqrt{n} \leq p$, a.s.
 \[
 \Gamma(G(n, p)) \leq O\left(np^{2/3}\right).
 \]
 
\end{corollary} 
\section{Conclusion}

We presented two spectral upper bounds to the Grundy number of a graph. The first is in terms of the largest eigenvalue of the corresponding graph alone; a result achieved by characterizing the graphs with the smallest largest eigenvalue among all $k$-atoms. The second bound takes into account the size of the $k$-atom, a result achieved by bounding the largest eigenvalue of a matrix $B$ obtained by interlacing on the adjacency matrix. Lastly, we applied these two results to bound the Grundy number of random graphs with non-constant edge density.

Theorem \ref{thm:grundy_bound} is tight for the largest $k$-atom, the $T_k$ tree, and Corollary \ref{corol:grundy_bound_2} is asymptotically tight for the smallest $k$-atom, which is the complete graph $K_k$. However, it is unknown weather the spectral bound \eqref{eq:degeneracy}, obtained from the degeneracy of the graph is tight or if it can be improved.

\begin{problem}
    Characterize a family of graphs $\mathcal{F}$ for which $\Gamma(G) \geq \Omega(\lambda_1(G) \log n)$ for all $G \in \mathcal{F}$ or improve inequality \eqref{eq:degeneracy}.
\end{problem}

The inequalities for the Grundy number presented in this paper offer only upper bounds on the Grundy number; however, a spectral lower bound that is not a lower bound for the chromatic number remains unknown. A primary challenge in finding a spectral lower bound is the fact that we may no longer be able to use the characterization in terms of induced $k$-atoms, as interlacing should yield an inequality in the opposite direction.

\begin{problem}
    Find a lower bound on $\Gamma(G)$ in terms of the spectrum of $G$ that is not (always) a lower bound for $\chi(G)$.
\end{problem}

\section*{Acknowledgements}

We acknowledge Michael Tait for conversations about the topic of this paper. All authors acknowledge the financial support from CNPq and FAPEMIG.

\bibliographystyle{plain}
\IfFileExists{references.bib}
{\bibliography{references.bib}}
{\bibliography{../references}}

	
\end{document}